\documentclass[a4paper,12pt]{amsart}

\usepackage{amscd}
\usepackage{amssymb, color}

\theoremstyle{plain}
  \newtheorem{thm}{Theorem}
  
  \newtheorem{lem}[thm]{Lemma}
  \newtheorem{cor}[thm]{Corollary}

  \newtheorem{thqt}{Theorem}

\theoremstyle{definition}

  \newtheorem{defi}{Definition}

\numberwithin{equation}{section}

\author[T. Yamazaki]{Takeaki Yamazaki}
\address{%
Department of Electrical, Electronic and Computer Engineering,
Toyo University,
Kawagoe-Shi, Saitama, 350-8585, Japan.
} \email{t-yamazaki@toyo.jp}

\keywords{Positive definite operator; operator mean;
operator monotone function; power mean; the 
Ando-Hiai inequality.}

\subjclass[2010]{Primary 47A64. Secondary 47A63.}

\title{An integral representation of operator means
via the power means and an application to the Ando-Hiai
inequality 
}

\begin{document}

\begin{abstract}
An integral representation of an operator mean 
via the power means is obtained. 
As an application, we shall give 
explicit condition of operator means that the 
Ando-Hiai inequality holds.
\end{abstract}

\maketitle

\section{Introduction}
The theory of operator means have been started in 
\cite{PW1975}. In this paper, the operator geometric mean
has been defined. Then the axiom of operator means was
given in \cite{KA1980}. In this discussion, Kubo and Ando
 obtained that there is a one-to-one 
 correspondence between operator means and 
 operator monotone functions.
This operator monotone function is called 
a representing function of an operator mean. 
Applying the Loewner's theorem,
each representing function of operator means has
an integral representation via the harmonic means.
Using such an integral representation,
we can obtain a lot of properties of operator means
by using the properties of the harmonic means.
However it is difficult to obtain individual properties of 
operator means by using the 
integral representation.
In this paper, we shall give an integral representation of
a representing function of an operator mean via the 
power mean. Because of 
the power mean interpolates the arithmetic, 
the geometric and the harmonic means,
this integral representation can be considered as 
a more precise result than the known 
result.
In fact, we shall give a property of operator means
which are greater than the geometric means,
and give an explicit condition of 
operator means that the Ando-Hiai inequality holds.

In what follows let $\mathcal{H}$ be a Hilbert space
with an inner product $\langle \cdot, \cdot\rangle$, 
and $\mathcal{B(H)}$ be a set of all 
bounded linear operators 
on $\mathcal{H}$.
An operator $A\in \mathcal{B(H)}$ is positive definite 
(resp. positive semi-definite)
if $\langle Ax,x\rangle >0$ 
(resp. $\langle Ax,x\rangle \geq 0$)
holds for all non-zero 
$x\in \mathcal{H}$.
If $A$ is positive semi-definite, we denote 
$A\geq 0$. 
Let $\mathcal{PS}, \mathcal{P}\subset 
\mathcal{B(H)}$ 
be the sets of all positive semi-definite
and positive definite operators, respectively.
For self-adjoint operators $A$ and $ B$, 
$A\geq B$ is defined by $A-B\geq 0$.
A real-valued function $f$ 
defined on an interval $I$ 
satisfying 
$$ B\leq A \ \Longrightarrow \ f(B)\leq f(A) $$
for all self-adjoint operators $A,B\in 
\mathcal{B(H)}$ such that
$\sigma(A), \sigma(B)\subset I$ 
is called an operator monotone function,
where $\sigma(X)$ means the spectrum of
$X\in \mathcal{B(H)}$.

\subsection{Operator mean}

\begin{defi}[Operator mean, \cite{KA1980}]
Let $\sigma: \mathcal{PS}^{2} \to \mathcal{PS}$ be
a binary operation. If $\sigma$ satisfies the 
following four conditions, then $\sigma$ is called
an operator mean.
\begin{itemize}
\item[(1)] If $A\leq C$ and $B\leq D$, then 
$\sigma(A,B)\leq \sigma(C,D)$,
\item[(2)] $X^{*}\sigma(A,B)X\leq 
\sigma(X^{*}AX, X^{*}BX)$ for all $X\in 
B(\mathcal{H})$,
\item[(3)] $\sigma$ is upper semi-continuous 
on $\mathcal{PS}^{2}$,
\item[(4)] $\sigma(I,I)=I$, where $I$ means the 
identity operator in $\mathcal{B(H)}$.
\end{itemize}
\end{defi}

We notice that if $X$ is invertible in (2),
then equality holds.

\begin{thqt}[\cite{KA1980}]
Let $\sigma$ be an operator mean.
Then there exists an operator monotone 
function $f$ on $(0, \infty)$ such that $f(1)=1$ and 
$$ \sigma(A,B)=A^{\frac{1}{2}}f(A^{\frac{-1}{2}}
BA^{\frac{-1}{2}})A^{\frac{1}{2}} $$
for all $A\in \mathcal{P}$ and $B\in \mathcal{PS}$.
A function $f$ is called a representing 
function of an operator mean $\sigma$.
\end{thqt}

Let $\varepsilon$ be a positive real number. Then
we have $A_{\varepsilon}=A+\varepsilon I,
B_{\varepsilon}=B+\varepsilon I\in \mathcal{P}$
for $A,B\in \mathcal{PS}$, 
and we can compute an operator mean 
$\sigma(A,B)$ by 
$ \sigma(A,B)=\lim_{\varepsilon\searrow 0}
\sigma(A_{\varepsilon},B_{\varepsilon}). $
We note that for an operator mean 
$\sigma$ with a representing function $f$,
$f'(1)=\lambda
\in [0,1]$ (cf. \cite{H2013}),
and we call $\sigma$ a $\lambda$-weighted operator
mean. 
Typical examples of operator means 
are the $\lambda$-weighted geometric  
and $\lambda$-weighted power means. These 
representing functions are $f(x)=x^{\lambda}$ and 
$f(x)=[1-\lambda+\lambda x^{t}]^{\frac{1}{t}}$, 
respectively, 
where $\lambda\in [0,1]$
and $t\in[-1,1]$ (in the case $t=0$, we consider 
$t\to 0$). The weighted power mean interpolates 
the arithmetic, the geometric and the harmonic means 
by putting $t=1,0,-1$, respectively.
In what follows, the $\lambda$-weighted geometric and 
$\lambda$-weighted power means 
of $A, B\in \mathcal{P}$ are denoted by
$A\sharp_{\lambda} B$ and 
$P_{t}(\lambda;A,B)$, respectively,
i.e., 
\begin{align*}
A\sharp_{\lambda}B  & 
= A^{\frac{1}{2}}(A^{\frac{-1}{2}}B
A^{\frac{-1}{2}})^{\lambda}A^{\frac{1}{2}},\\
 P_{t}(\lambda;A,B)  & =
A^{\frac{1}{2}}\left[1-\lambda+\lambda(A^{\frac{-1}{2}}B
A^{\frac{-1}{2}})^{t}\right]^{\frac{1}{t}}A^{\frac{1}{2}}.
\end{align*}
In what follows $p_{t}$ denotes the representing 
function of the power mean $P_{t}$, i.e., 
$$ p_{t}(\lambda;x)=[1-\lambda+\lambda x^{t}
]^{\frac{1}{t}}. $$
$p_{t}$ 
is monotone increasing on a parameter $t\in [-1,1]$.
The Loewner's theorem gives us 
an integral representation of an operator monotone 
function on $(0, \infty)$ such that $f(1)=1$
as follows.
\begin{thqt}[cf. \cite{H2013}]\label{thqt:Loewner}
Let $\sigma$ be an operator mean,
and $f_{\sigma}$ be a 
representing function of $\sigma$.
Then there exists a provability measure $d\mu$
on $[0,1]$ such that
$$ f_{\sigma}(x)  =\int_{0}^{1} p_{-1}(\lambda; x) d\mu(\lambda) 
 =\int_{0}^{1} [(1-\lambda)+\lambda x^{-1}]^{-1} 
d\mu(\lambda). 
$$
\end{thqt}

It is a very useful result, and can give us a lot of 
properties of operator means.
However it is difficult to give us 
individual properties of operator means.
According to Theorem \ref{thqt:Loewner}, 
we have $f'(1)=\int_{0}^{1}\lambda d\mu(\lambda)$.
Using the harmonic -arithmetic mean inequality, we 
have
\begin{equation}
\begin{split}
 f(x) & =
\int_{0}^{1} [(1-\lambda)+\lambda x^{-1}]^{-1} 
d\mu(\lambda) \\
& \leq 
\int_{0}^{1} [(1-\lambda)+\lambda x] 
d\mu(\lambda) 
= 1-f'(1)+f'(1)x 
\end{split}
\label{MA-inequality}
\end{equation}
holds for all $x\in (0,\infty)$, 

\medskip

The first aim of this paper is to 
generalize Theorem \ref{thqt:Loewner} into 
more precise form by using the power means.

\subsection{The Ando-Hiai inequality}
The Ando-Hiai inequality is one of the most 
important operator inequalities.

\begin{thqt}[The Ando-Hiai inequality, \cite{AH1994}]
\label{thqt:AH}
Let $A,B\in \mathcal{PS}$. If for each $t\in [0,1]$,
$A\sharp_{t}B\leq I$, then 
$A^{r}\sharp_{t}B^{r}\leq I$ holds for all 
$r\geq 1$.
\end{thqt}

This result only discusses the 
operator geometric means,
and it has been generalized in \cite{W2014, Y2012},
for example.
Especially, in \cite{W2014}, 
Wada obtained an equivalent condition 
of operator means that Theorem \ref{thqt:AH}
holds.

\begin{thqt}[\cite{W2014}]\label{thqt:W}
Let $\sigma$ be an operator mean with 
a representing function $f_{\sigma}$.
Then the following conditions are equivalent:
\begin{itemize}
\item[{\rm (1)}] $f_{\sigma}(x)^{r}\leq 
f_{\sigma}(x^{r})$ for all $r\geq 1$ 
and $x\in (0,\infty)$,
\item[{\rm (2)}] $\sigma(A,B)\leq I$ implies 
$\sigma(A^{r},B^{r})\leq I$ for all $r\geq 1$ and
$A,B\in \mathcal{PS}$.
\end{itemize}
\end{thqt}

Wada focused on the condition
(1) in \cite{W2014}, and he defined two sets of 
operator monotone functions as follows.
Let $\mathcal{M}$ be a set of all operator monotone 
functions on $(0,\infty)$ such that $f(1)=1$ 
for $f\in \mathcal{M}$, i.e., $\mathcal{M}$ is a set of all 
representing functions of all operator means. Let
\begin{align*}
PMI & = \{ f\in \mathcal{M} |\  f(x)^{r}\leq f(x^{r})
\text{ for all } r\geq 1 
\text{ and } x\in (0,\infty)\},\\
PMD & = \{ f\in \mathcal{M} |\  f(x)^{r}\geq f(x^{r})
\text{ for all } r\geq 1 
\text{ and } x\in (0,\infty)\}.
\end{align*}
If $f\in PMI$, then the correspondence operator mean
satisfies (2) in Theorem \ref{thqt:W}, and if
$f\in PMD$, then the correspondence operator mean
satisfies the opposite inequality of (2) in 
Theorem \ref{thqt:W}.
See \cite{W2014} for more information. 
The second aim of this paper is 
to give a explicit condition of operator monotone
functions in $PMI$ or $PMD$.

This paper is organized as follows:
In Section 2, we shall  
give an integral representation of 
a function in $\mathcal{M}$ via representing 
functions of the power means.
In Section 3, we shall give an explicit condition of 
operator means that 
Theorem \ref{thqt:W} (2) holds.

\section{Integral representation}

In this section, we shall give an 
integral representation of 
a function in $\mathcal{M}$ via the power means.
For $t\in[-1,1]$, let
$$C_{t}=\{ f \in \mathcal{M}|\ 
p_{t}(f'(1); x) 
\leq f(x) \text{ for all $x\in (0,\infty)$}\} .$$

\begin{thm}\label{thm:integral representation}
For each $t\in [-1,1]$ and $f\in C_{t}$,
there exists a probability measure $d\mu$
on $[0,1]$ such that
$$ f(x)=\int_{0}^{1} p_{t}(\lambda; x) 
d\mu(\lambda). $$ 
\end{thm}

\begin{proof}
First of all, the case of $t=1$ is obvious because 
$C_{1}=\{ p_{1}(\lambda; x) |\ \lambda\in[0,1]\}$
by \eqref{MA-inequality}.
Hence we only consider $t\in [-1,1)$.
We notice that $C_{t}$ is a closed set 
in the topology 
of point-wise convergence of functions.
$\mathcal{M}=C_{-1}$ is 
compact (cf. \cite{H2013,HP2014}).
Hence $C_{t}$ is closed and compact since 
$C_{t}\subseteq C_{-1}$.

We shall prove that $C_{t}$ is a convex 
set. Let $f,g\in C_{t}$. Then
for each $\alpha\in [0,1]$, 
\begin{align*}
p_{t}((1-\alpha)f'(1)+\alpha g'(1);x)
& \leq
(1-\alpha)p_{t}(f'(1);x)+
\alpha p_{t}(g'(1);x)\\
& \leq 
(1-\alpha)f(x)+
\alpha g(x)
\end{align*}
holds for all $x\in (0,\infty)$
since $p_{t}(\lambda;x)$ is convex on 
$\lambda\in [0,1]$.
Hence $(1-\alpha)f(x)+
\alpha g(x)\in C_{t}$, and $C_{t}$ is a convex set.
Therefore $C_{t}$ is a convex, closed and compact
set. By the Krein-Milmann's theorem, we have 
$C_{t}={\rm conv}(\overline{{\rm Ext}(C_{t})})$,
where ${\rm conv}(X)$ and ${\rm Ext}(X)$ are 
convex hull and the set of all extreme points of a set $X$, 
respectively.

Next, we shall show that every $p_{t}$ 
is an extreme point of 
$C_{t}$. Assume that
\begin{align}
 \alpha f(x)+(1-\alpha) g(x)= p_{t}(\mu; x) 
\label{proof:thm4-1}
\end{align}
holds for some $f, g\in C_{t}$, $\mu\in [0,1]$ and all 
$x\in (0,\infty)$. Then 
\begin{equation}
\begin{split}
p_{t}(\mu; x) & =
\alpha f(x)+(1-\alpha) g(x) \\
& \geq 
\alpha  p_{t}(f'(1); x)+(1-\alpha) 
p_{t}(g'(1); x)\\
& \geq 
 p_{t}(\alpha f'(1)+(1-\alpha)g'(t); x).
\end{split}
\label{proof:thm:integral representation-1}
\end{equation}
%
%
Since $\mu=p'_{t}(\mu;1)$ and
 \eqref{proof:thm4-1}, we have
$ \mu=\alpha f'(1)+(1-\alpha) g'(1)$, and 
by \eqref{proof:thm:integral representation-1}, 
we have
$$ \alpha p_{t}(f'(1); x)+(1-\alpha)
p_{t}(g'(1); x)=
p_{t}(\mu; x). $$
Since $p_{t}(\mu; x)$ is strictly convex on 
$\mu\in [0,1]$, we have $\mu=f'(1)=g'(1)$, and 
$p_{t}(\mu; x)=f(x)=g(x)$. 
Hence $p_{t}(\mu; x)$ is an extreme point of 
$C_{t}$.

Therefore each element in $C_{t}$ can be 
representing by a convex combination of 
$p_{t}$.
Hence Theorem 
\ref{thm:integral representation} is 
completed.
\end{proof}

Especially, if $t\to 0$, then we have 
the following corollary.

\begin{cor}
For each $f\in C_{0}$,
there exists a probability measure $d\mu$
on $[0,1]$ such that
$$ f(x)=\int_{0}^{1} x^{\lambda} d\mu(\lambda). $$ 
\end{cor}

We remark that Theorem 
\ref{thm:integral representation}
can be extended from $p_{t}$ 
into any representing function of an operator mean
such that it is strict convex on the weight parameter.
The weighted operator mean is obtained 
by using the algorithms in \cite{FK1989} and 
\cite{PP2014}.
However the author does not know any 
operator mean which is strict convex on the
weight parameter except the power means.

\section{An application to the Ando-Hiai inequality}

In this section, we shall discuss the Ando-Hiai  
inequality. Especially, we shall give concrete examples
some operator means which satisfy the 
Ando-Hiai inequality.

\begin{thm}\label{thm:Ando-Hiai}
Let $\sigma$ be an operator mean with 
a representing function $f_{\sigma}$.
Then the following conditions are equivalent:
\begin{itemize}
\item[{\rm (1)}] $x^{f'_{\sigma}(1)} \leq 
f_{\sigma}(x)$ for all $x\in (0,\infty)$,
\item[{\rm (2)}] $\sigma(A,B)\leq I$ implies 
$\sigma(A^{r},B^{r})\leq I$ for all $r\geq 1$ and
$A,B\in \mathcal{PS}$.
\end{itemize}
\end{thm}

Combined with Theorems \ref{thqt:W} and
\ref{thm:Ando-Hiai}, we can get
$$ PMI  = \{ f\in M |\  x^{f'(1)} \leq f(x)
\text{ for all  } x\in (0,\infty)\}.$$
To prove Theorem \ref{thm:Ando-Hiai}, 
we shall give the following lemma.

\begin{lem}\label{lem:equivalent}
Let $f\in \mathcal{M}$.
Then the following conditions are equivalent.
\begin{itemize}
\item[{\rm (1)}] $f(x)^{r}\leq f(x^{r})$ holds
for all $r\geq 1$ and $x\in (0,\infty)$,
\item[{\rm (2)}] $x^{f'(1)} \leq f(x)$ holds
for all $x\in (0,\infty)$.
\end{itemize}
\end{lem}

\begin{proof}
Proof of (1) $\Longrightarrow$ (2).
(1) is equivalent to $f(x^{p})^{\frac{1}{p}}\leq 
f(x)$ for all $p\in (0,1]$ and $x\in (0,\infty)$.
Then we have $\lim_{p\to 0}f(x^{p})^{\frac{1}{p}}\leq 
f(x)$. Here by the H'lospital's theorem, we have
$$ \lim_{p\to 0}\log f(x^{p})^{\frac{1}{p}}=
\lim_{p\to 0}\frac{\log f(x^{p})}{p}=
\lim_{p\to 0}\frac{f'(x^{p})x^{p}\log x}{f(x^{p})}
=\log x^{f'(1)}. $$
Hence we have (2).

Proof of (2) $\Longrightarrow$ (1).
If $f$ satisfies (2), then $f\in C_{0}$.
By Theorem \ref{thm:integral representation}, 
there exists a probability measure $d\mu$ 
on $[0,1]$ such that
$$ f(x)=\int_{0}^{1}x^{\lambda}d\mu(\lambda). $$
Since $t^{\frac{1}{p}}$ is a convex function 
for $p\in (0,1)$, we have
$$ f(x^{p})^{\frac{1}{p}} =
\left(\int_{0}^{1}x^{p\lambda}dE_{\lambda}
\right)^{\frac{1}{p}}
\leq 
\int_{0}^{1}x^{\lambda}dE_{\lambda}=f(x). $$
Therefore the proof is completed.
\end{proof}

\begin{proof}[Proof of Theorem \ref{thm:Ando-Hiai}]
By Lemma \ref{lem:equivalent}, (1) of 
Theorems \ref{thqt:W} and \ref{thm:Ando-Hiai}
are equivalent. Hence the proof is completed.
\end{proof}

Using Theorem \ref{thm:Ando-Hiai},
we can obtain examples of operator means 
such that Theorem \ref{thm:Ando-Hiai} (2) 
holds.

\medskip
\noindent
{\bf Examples.}
\begin{itemize}
\item[(1)] The Logarithmic mean:
The representing function is 
$f(x)=\frac{x-1}{\log x}$, and it satisfies 
$\sqrt{x}\leq f(x)\leq \frac{1+x}{2}$ for all 
$x\in (0, \infty)$.
Hence the operator logarithmic mean $L(A,B)$ 
satisfies that $L(A,B)\leq I$ implies 
$L(A^{r}, B^{r})\leq I$ for all $r\geq 1$ and
$A,B\in \mathcal{PS}$.

\medskip

\item[(2)] The identric mean:
The representing function is
$f(x)=\exp(\frac{x\log x}{x-1}-1)$.
It can be obtained as $F_{1,0}(x)$, where
$$ F_{p,q}(x)=\left(\int_{0}^{1} [1-\lambda +
\lambda x^{p}]^{\frac{q}{p}}d\lambda
\right)^{\frac{1}{q}}=
\left(\frac{p}{p+q}\frac{x^{p+q}-1}{x^{p}-1}\right)
^{\frac{1}{q}}.$$
$F_{p,q}(x)$ is a representing function of 
the extension of the power difference mean.
In \cite{UWYY2015}, $F_{p,q}$ is monotone increasing
on each parameter $p,q\in [-1,1]$ 
(the cases $p=0$ and $q=0$ are defined by 
taking limits). Then we have
\begin{align*}
\frac{1+x}{2} & = F_{1,1}(x) \\
& \geq 
F_{1,0}(x) =\exp\left(\frac{x\log x}{x-1}-1\right)\\
& \geq 
F_{0,0}(x)=\sqrt{x}.
\end{align*}
Hence the operator identric mean $\mathcal{I}(A,B)$ 
satisfies that $\mathcal{I}(A,B)\leq I$ implies 
$\mathcal{I}(A^{r}, B^{r})\leq I$ for all $r\geq 1$ and
$A,B\in \mathcal{PS}$.

\medskip

\item[(3)] The Heinz mean:
For $t\in [0,1]$, the Heinz mean $M_{t}$ is
defined as follows:
$$ M_{t}(A,B)=\frac{A\sharp_{t}B+A\sharp_{1-t}B}{2}. $$
It is easy that 
$$ A\sharp B\leq M_{t}(A,B)\leq \frac{A+B}{2}. $$
Hence $M_{t}(A,B)\leq I$ implies 
$M_{t}(A^{r}, B^{r})\leq I$ for all $r\geq 1$ and
$A,B\in \mathcal{PS}$.
\end{itemize}

\medskip

In \cite{W2014}, it is shown that
$f(x)\in PMI$ if and only if 
$f^{*}(x):=f(x^{-1})^{-1}, f^{\perp}(x):=
x f(x)^{-1}\in PMD$. Hence we have the following.
$$ PMD  = \{ f\in M | \ x^{f'(1)} \geq f(x)
\text{ for all  } x\in (0,\infty)\}.$$

\medskip
\noindent
{\bf Theorem \ref{thm:Ando-Hiai}'.}
{\it Let $\sigma$ be an operator mean with 
a representing function $f_{\sigma}$.
Then the following conditions are equivalent:}
\begin{itemize}
\item[{\rm (1)}] $x^{f'_{\sigma}(1)} \geq 
f_{\sigma}(x)$ {\it for all} $x\in (0,\infty)$,
\item[{\rm (2)}] $\sigma(A,B)\geq I$ {\it implies} 
$\sigma(A^{r},B^{r})\geq I$ {\it for all $r\geq 1$ and}
$A,B\in \mathcal{PS}$.
\end{itemize}

\end{document}